\documentclass[11pt]{amsart}

\usepackage[margin=1in]{geometry}
\usepackage{amsmath,amssymb,amsthm}
\usepackage{enumitem}
\usepackage{xcolor}
\usepackage{tikz}
\usepackage[colorlinks=true,linkcolor=blue!55!black,citecolor=blue!55!black,urlcolor=blue!55!black]{hyperref}
\usepackage[nameinlink,capitalize]{cleveref}
\usepackage{microtype}

\newtheorem{theorem}{Theorem}[section]
\newtheorem{proposition}[theorem]{Proposition}
\newtheorem{lemma}[theorem]{Lemma}
\newtheorem{corollary}[theorem]{Corollary}
\theoremstyle{definition}
\newtheorem{example}[theorem]{Example}
\newtheorem{remark}[theorem]{Remark}

\newcommand{\reg}{\operatorname{reg}}
\newcommand{\cl}{\operatorname{cl}}
\newcommand{\ex}{\operatorname{ex}}
\newcommand{\conv}{\operatorname{conv}}
\newcommand{\m}{\mathfrak m}
\newcommand{\NN}{\mathbb N}
\newcommand{\ZZ}{\mathbb Z}
\newcommand{\kk}{\mathbb K}
\newcommand{\cQ}{\mathcal Q}
\newcommand{\cL}{\mathcal L}
\newcommand{\cI}{\mathcal I}
\newcommand{\doi}[1]{\href{https://doi.org/#1}{doi:\nolinkurl{#1}}}

\title{Regularity of Symbolic Powers of Co-Chordal Edge Ideals}

\author{Chwas Ahmed\textsuperscript{1,2}}
\address{\textsuperscript{1} Department of Mathematics, College of Science, University of Sulaimani, Sulaimani, Kurdistan Region, Iraq.}
\address{\textsuperscript{2} Department of Mathematics and Natural Sciences, American University of Iraq, Sulaimani, Kurdistan Region, Iraq.}
\email{chwas.ahmed@univsul.edu.iq}

\author{Mohammed Rafiq Namiq\textsuperscript{1}}
\address{\textsuperscript{1} Department of Mathematics, College of Science, University of Sulaimani, Sulaimani, Kurdistan Region, Iraq.}
\email{mohammed.namiq@univsul.edu.iq}

\begin{document}

\begin{abstract}
Let $G$ be a finite simple co-chordal graph with at least one edge, and let $I(G)$ be its edge ideal. Over an arbitrary field, we prove that the symbolic powers of $I(G)$ satisfy
\[
\reg I(G)^{(s)}=2s
\qquad\text{for every }s\ge1.
\]
Thus every symbolic power of $I(G)$ has a degree resolution. Using Takayama's formula and clique trees, we reduce the regularity problem to a topological one. Finite convex geometry then shows that nontrivial homology implies a suitable set of leaves, and a weighted counting argument gives the required regularity formula. Finally, we show that symbolic powers of co-chordal edge ideals are not necessarily componentwise linear.
\end{abstract}

\maketitle

\section{Introduction}

Let $\kk$ be a field and let $S=\kk[x_1,\ldots,x_n]$ be a standard graded polynomial ring. For a nonzero finitely generated graded $S$-module $M$, write $\beta_{i,j}(M)=\dim_\kk \operatorname{Tor}_i^S(M,\kk)_j$ for its graded Betti numbers. The Castelnuovo--Mumford regularity of $M$ is
\[
\reg M=\max\{j-i:\beta_{i,j}(M)\ne0\}.
\]

For a homogeneous ideal $I\subseteq S$, classical results show that the regularity of the ordinary powers $I^s$ is eventually linear in $s$ \cite{CutkoskyHerzogTrung,Kodiyalam}. Symbolic powers can behave differently. For monomial ideals, the regularity of symbolic powers is eventually quasi-linear and need not be eventually linear \cite{DungHienNguyenTrung,HerzogHibiTrung}. This leads naturally to the problem of determining classes of combinatorially defined ideals for which the regularity of symbolic powers has a particularly simple form.

Let $G$ be a finite simple graph on $[n]$, and let
\[
I(G)=(x_ix_j:\{i,j\}\in E(G))
\]
be its edge ideal \cite{villareal}. Since $I(G)$ is squarefree, its symbolic powers are determined by the minimal vertex covers of $G$. The regularity of symbolic powers of edge ideals has been studied for several classes of graphs, including bipartite, unicyclic, chordal, Cameron--Walker, and cubic circulant graphs \cite{FakhariCameronWalker,FakhariUnicyclic,FakhariChordal,HangPhamVu,SimisVasconcelosVillarreal}.

In this paper, we consider co-chordal graphs, that is, graphs whose complements are chordal. This class is especially natural from the viewpoint of regularity. By Fr\"oberg's theorem, $I(G)$ has a $2$-linear resolution if and only if $G^c$ is chordal \cite{Froberg}. Herzog, Hibi, and Zheng proved that all powers of a monomial ideal with a $2$-linear resolution also have linear resolutions \cite{HHZ}. Consequently, if $G$ is co-chordal and has at least one edge, then
\[
\reg I(G)^s=2s
\qquad\text{for every }s\ge1.
\]
It is therefore natural to ask whether the symbolic powers have the same regularity.

Several results are closely related to this question. As a consequence of his formula for chordal graphs, Seyed Fakhari obtained $\reg I(G)^{(s)}=2s$ for every split graph with at least one edge \cite[Theorem~3.3]{FakhariChordal}, since split graphs are both chordal and co-chordal and have induced matching number one. For co-chordal graphs, he also proved the same equality for $s=2,3,4$ \cite{FakhariSmall}. Ficarra, Moradi, and R\"omer established the equality for complements of block graphs and showed that the largest degree of a minimal generator of $I(G)^{(s)}$ is $2s$ for every co-chordal graph \cite{FMR}. The latter result gives no corresponding bound on higher syzygies and therefore does not suffice to determine $\reg I(G)^{(s)}$ for arbitrary $s$.

Our main result gives the exact value for every symbolic power of an arbitrary co-chordal edge ideal.

\begin{theorem}\label{thm:1}
Let $G$ be a finite simple co-chordal graph with at least one edge. Then, over every field,
\[
\reg I(G)^{(s)}=2s
\qquad\text{for every }s\ge1.
\]
\end{theorem}

For a monomial ideal $L\subseteq S$ and an integer $d\ge0$, let $L_d$ denote the degree $d$ component of $L$, and let $L_{\langle d\rangle}$ denote the ideal generated by $L_d$. Recall that $L$ is componentwise linear if $L_{\langle d\rangle}$ has a $d$-linear resolution for every $d$. The following result describes the ideals generated by these degree components, including the first possible nonzero one and those covered by the regularity theorem.

\begin{theorem}\label{thm:1.2}
Let $G$ be a finite simple co-chordal graph with at least one edge, let $s\ge1$, and put $J=I(G)^{(s)}$. Then the following statements hold.
\begin{enumerate}[label=\textup{(\roman*)}]
\item $J_{\langle d\rangle}=0$ for $d\le s$.
\item $J_{\langle s+1\rangle}$ is either zero or has an $(s+1)$-linear resolution.
\item $J_{\langle d\rangle}$ has a $d$-linear resolution for every $d\ge2s$.
\end{enumerate}
Thus only the degrees $s+2\le d\le2s-1$ can remain undecided, and this interval is nonempty only for $s\ge3$.
\end{theorem}

The lower bound is proved in \cref{lem:1}, where we show that $x_u^s x_v^s$ is a minimal generator of $I(G)^{(s)}$ for every edge $\{u,v\}\in E(G)$. Thus the main point of \cref{thm:1} is the upper bound. For a nonzero monomial ideal $J$, let $\omega(J)$ denote the largest degree of a minimal generator of $J$. Following \cite{Namiq}, we say that $J$ has a degree resolution if $\reg J=\omega(J)$. Since \cref{lem:1} gives a minimal generator of degree $2s$ and \cref{thm:1} gives $\reg I(G)^{(s)}=2s$, it follows that $I(G)^{(s)}$ has a degree resolution for every $s\ge1$. Together with the ordinary power equality above, \cref{thm:1} confirms Minh's proposed equality $\reg I(G)^{(s)}=\reg I(G)^s$ for co-chordal graphs \cite{GuHaORourkeSkelton,MinhVuSurvey}.

At the end of the paper, we show that for every $s\ge4$, if $H$ is the $s$-sun and $G=H^c$, then the ideal generated by the degree $s+2$ component of $I(G)^{(s)}$ does not have a linear resolution. Hence $I(G)^{(s)}$ is not componentwise linear. Consequently, \cref{thm:1} cannot be strengthened to a componentwise linearity statement. This also disproves \cite[\text{Conjecture~B}]{FMR}.

We briefly describe the proof of Theorem~\ref{thm:1}. Put $H=G^c$. Since $H$ is chordal, its maximal cliques admit a clique tree \cite{BlairPeyton}. Takayama's formula translates multigraded local cohomology into reduced homology of degree complexes \cite{CHHKTT,Takayama}, which are described in \cref{prop:1}. The clique tree then allows us to replace these degree complexes, up to homotopy, by trace complexes of subtrees, as shown in \cref{prop:2}.

To analyze the homology of the trace complexes, we use finite convex geometries. The homological criterion in \cref{thm:2} produces a free closed nonface of cardinality $q+2$ from nonzero reduced homology in degree $q$. For tree convexity, \cref{lem:6} identifies the extreme elements of a set with the leaves of the minimal subtree containing it. Thus the homological witness becomes a leaf configuration in the tree model. The weighted subtree theorem \cref{thm:3} converts this configuration into the numerical estimate required for the regularity bound. The complementary case is treated in \cref{lem:4}. Together these results yield the desired upper bound $\reg I(G)^{(s)}\le2s$ in \cref{thm:1}.

The paper is organized as follows. In \cref{sec:1} we recall symbolic powers of edge ideals and prove the general lower bound. In \cref{sec:2} we introduce Takayama degree complexes and their connection with multigraded local cohomology. In \cref{sec:3} we develop the homological criterion for convex geometries. In \cref{sec:4} we construct the clique tree trace model and prove the weighted subtree inequality. These results are combined in \cref{sec:5} to prove \cref{thm:1} and its consequences.

\section{Symbolic powers}\label{sec:1}

Let $G$ be a finite simple graph on $[n]$, let $H=G^c$ be its complement, and let $\cQ(H)$ denote the set of maximal cliques of $H$. A maximal clique of $H$ is a maximal independent set of $G$, and its complement is a minimal vertex cover of $G$. For $Q\in\cQ(H)$, put $C_Q=[n]\setminus Q$ and $P_{C_Q}=(x_i:i\in C_Q)$. Then
\begin{equation}\label{eq:1}
I(G)^{(s)}=\bigcap_{Q\in\cQ(H)}P_{C_Q}^s.
\end{equation}

\begin{lemma}\label[lemma]{lem:1}
If $G$ has an edge, then
\[
\reg I(G)^{(s)}\ge 2s
\qquad\text{for every }s\ge1.
\]
Moreover, if $\{u,v\}\in E(G)$, then $x_u^s x_v^s$ is a minimal monomial generator of $I(G)^{(s)}$.
\end{lemma}
\begin{proof}
Fix an edge $\{u,v\}\in E(G)$ and put $m=x_u^s x_v^s$. Every minimal vertex cover $C$ of $G$ meets $\{u,v\}$, so the total exponent of $m$ on $C$ is at least $s$. By \cref{eq:1}, $m\in I(G)^{(s)}$. To prove minimality, extend $\{v\}$ to a maximal independent set $Q_v$. Since $u$ is adjacent to $v$, the minimal vertex cover $C_v=[n]\setminus Q_v$ contains $u$ but not $v$. Hence the total exponent of $x_u^{s-1}x_v^s$ on $C_v$ is $s-1$, and therefore $x_u^{s-1}x_v^s\notin I(G)^{(s)}$. Interchanging $u$ and $v$ gives $x_u^s x_v^{s-1}\notin I(G)^{(s)}$. Since $I(G)^{(s)}$ is a monomial ideal, these two exclusions show that $m$ is a minimal generator. Thus $I(G)^{(s)}$ has a minimal generator of degree $2s$. A minimal generator of degree $d$ gives $\beta_{0,d}\ne0$, so $\reg I(G)^{(s)}\ge2s$.
\end{proof}

The same lower bound follows from the stronger induced matching bound
\[
\reg I(G)^{(s)}\ge 2s+\nu(G)-1,
\]
where $\nu(G)$ denotes the induced matching number of $G$ \cite[Theorem~4.6]{GuHaORourkeSkelton}. We give a direct proof because it also identifies an explicit minimal generator. It remains to prove $\reg I(G)^{(s)}\le2s$. Equivalently, since $I(G)^{(s)}$ is a nonzero proper monomial ideal, it suffices to prove $\reg S/I(G)^{(s)}\le2s-1$.

\section{Takayama degree complexes}\label{sec:2}

We use $\NN=\{0,1,2,\dots\}$. Let $J\subseteq S$ be a monomial ideal and let $\mathbf a=(a_1,\dots,a_n)\in\ZZ^n$. Write $|\mathbf a|=\sum_{j=1}^n a_j$ and $N=N(\mathbf a)=\{i:a_i<0\}$. For $U\subseteq[n]$, write $S_U=S[x_j^{-1}:j\in U]$. For $F\subseteq[n]\setminus N$, the Takayama degree complex is
\[
\Delta_{\mathbf a}(J)
=
\bigl\{F\subseteq[n]\setminus N:
 x^{\mathbf a}\notin JS_{F\cup N}\bigr\}.
\]
For subsets $F_1,\dots,F_t\subseteq[n]$, the notation $\langle F_1,\dots,F_t\rangle$ means the simplicial complex consisting of all subsets of the $F_i$. Its facets are the sets among the $F_i$ that are maximal under inclusion. We use reduced simplicial homology with the standard conventions. In particular, the order complex of the empty poset is $\{\varnothing\}$, $\widetilde H_{-1}(\{\varnothing\};\kk)=\kk$, and the reduced homology of the void complex vanishes in every degree. Relative homology is computed using the same conventions.

Takayama's formula \cite[Theorem~1]{Takayama} gives the equality of dimensions
\begin{equation}\label{eq:2}
\dim_\kk H^\ell_{\m}(S/J)_{\mathbf a}
=
\dim_\kk\widetilde H_{\ell-|N|-1}
  (\Delta_{\mathbf a}(J);\kk).
\end{equation}
We also use the multigraded local cohomology characterization \cite[Proposition~2.5(ii)]{CHHKTT}
\begin{equation}\label{eq:3}
\reg S/J
=
\max\bigl\{
|\mathbf a|+\ell:
H^\ell_{\m}(S/J)_{\mathbf a}\ne0
\bigr\}.
\end{equation}

Write $\operatorname{Cl}(H)$ for the clique complex of $H$. The next result uses \cite[Lemma~1.3]{HoaTrung2016}, which is stated for arbitrary Stanley--Reisner ideals, to describe degree complexes of symbolic powers. We only need to translate the facets of the relevant link into maximal cliques of $H$.

\begin{proposition}\label[proposition]{prop:1}
Let $J=I(G)^{(s)}$ and $N=N(\mathbf a)$. If $\Delta_{\mathbf a}(J)$ is nonvoid, then
\begin{equation}\label{eq:4}
\Delta_{\mathbf a}(J)
=
\left\langle
Q\setminus N:
Q\in\cQ(H),\ N\subseteq Q,\sum_{i\notin Q}a_i\le s-1
\right\rangle.
\end{equation}
\end{proposition}
\begin{proof}
Set $\Delta=\operatorname{Cl}(H)$. Its faces are the independent sets of $G$, so $I_\Delta=I(G)$ and $J=I_\Delta^{(s)}$. Since $\Delta_{\mathbf a}(J)$ is nonvoid, $N$ is a face of $\Delta$. Indeed, if $N$ contained an edge $\{u,v\}$ of $G$, then $(x_ux_v)^s\in I(G)^s\subseteq J$ would become a unit in $S_N$. Hence $JS_N=S_N$, and every further localization $JS_{F\cup N}$ would be the whole ring, making $\Delta_{\mathbf a}(J)$ void. We claim that
\[
\operatorname{lk}_{\operatorname{Cl}(H)}(N)
=
\left\langle Q\setminus N:
Q\in\cQ(H),\ N\subseteq Q\right\rangle.
\]
If $F$ is a face of the link, then $F\cup N$ is a clique of $H$ and is contained in a maximal clique $Q$, so $F\subseteq Q\setminus N$. Conversely, if $N\subseteq Q$ and $F\subseteq Q\setminus N$, then $F\cup N\subseteq Q$ is a clique, so $F$ lies in the link. Since the $Q$ are maximal cliques, the facets of the link are exactly the sets $Q\setminus N$ with $N\subseteq Q$. By \cite[Lemma~1.3]{HoaTrung2016}, the facets of $\Delta_{\mathbf a}(I_\Delta^{(s)})$ are the facets $F$ of this link for which $\sum_{i\notin F\cup N}a_i\le s-1$. Writing $F=Q\setminus N$ gives $F\cup N=Q$, so this condition becomes $\sum_{i\notin Q}a_i\le s-1$. Thus the facets are precisely those listed in \cref{eq:4}, which proves the formula.
\end{proof}

Replacing each negative coordinate of $\mathbf a$ by $-1$ preserves $N$ and the complex in \cref{eq:4}. Indeed, every maximal clique occurring in that formula contains $N$, so each sum $\sum_{i\notin Q}a_i$ involves only coordinates outside $N$. Thus changing the coordinates indexed by $N$ does not change any admissibility condition. By \cref{eq:2}, the corresponding local cohomology component remains nonzero, while replacing a negative coordinate by $-1$ weakly increases the total degree. Hence, when proving an upper bound in \cref{eq:3}, we may assume $a_i=-1$ for $i\in N$. Put
\begin{equation}\label{eq:5}
b_j=a_j\quad(j\notin N),
\qquad
W=\sum_{j\notin N}b_j,
\qquad
q=\ell-|N|-1,
\qquad
r=W-s+1.
\end{equation}
Then
\begin{equation}\label{eq:6}
|\mathbf a|+\ell=W+q+1.
\end{equation}
For a maximal clique $Q\supseteq N$, the admissibility condition in \cref{eq:4} becomes
\begin{equation}\label{eq:7}
\sum_{j\in Q\setminus N}b_j\ge r.
\end{equation}

\section{Convex geometry and homological witnesses}\label{sec:3}

We next prove the abstract homological statement needed for the case $r\ge1$ and recall the finite convex geometry terminology used below. A finite convex geometry $(E,\cl_E)$ consists of a finite set $E$ with an extensive, monotone, and idempotent closure operator such that $\cl_E(\varnothing)=\varnothing$ and the anti-exchange axiom holds. For $X\subseteq E$ and distinct $x,y\notin\cl_E(X)$, $x\in\cl_E(X\cup\{y\})$ implies $y\notin\cl_E(X\cup\{x\})$. For $A\subseteq E$, the restriction to $A$ is the closure operator $\cl_A(X):=A\cap\cl_E(X)$ for $X\subseteq A$. The restriction is again a convex geometry. We have $\cl_A(\varnothing)=\varnothing$, and extensivity and monotonicity are immediate. Since $X\subseteq\cl_A(X)\subseteq\cl_E(X)$, monotonicity and idempotence give $\cl_E(\cl_A(X))=\cl_E(X)$, and hence $\cl_A(\cl_A(X))=\cl_A(X)$. For anti-exchange, let distinct $x,y\in A\setminus\cl_A(X)$ and suppose $x\in\cl_A(X\cup\{y\})$. Then $x,y\notin\cl_E(X)$ and $x\in\cl_E(X\cup\{y\})$. Anti-exchange for $\cl_E$ gives $y\notin\cl_E(X\cup\{x\})$, hence $y\notin\cl_A(X\cup\{x\})$.

For $C\subseteq E$, an element $c\in C$ is \emph{extreme} if $c\notin\cl_E(C\setminus\{c\})$. We write 
\[
\ex_E(C):=\{c\in C:c\notin\cl_E(C\setminus\{c\})\}.
\]
When the underlying closure operator is clear, we write simply $\cl$ and $\ex$. If $C\subseteq A$, then $\ex_A(C)=\ex_E(C)$, because for $c\in C\subseteq A$, $c\in\cl_A(C\setminus\{c\})$ if and only if $c\in\cl_E(C\setminus\{c\})$. For a closed set $C$, the finite Krein--Milman property gives $C=\cl_E(\ex_E(C))$ \cite{EdelmanJamison}. A closed set $C$ is \emph{free} if $\ex_E(C)=C$.

A simplicial complex $\mathcal K$ on $E$ is \emph{closure stable} with respect to $\cl$ if $\sigma\in\mathcal K$ implies $\cl(\sigma)\in\mathcal K$. For a finite poset $\mathcal P$, its order complex $\Delta(\mathcal P)$ is the simplicial complex whose faces are the nonempty chains of $\mathcal P$.

\begin{lemma}\label[lemma]{lem:2}
Let $\cL=\{C\subseteq E:\cl(C)=C\}$ be the lattice of closed sets of the finite convex geometry $(E,\cl)$, ordered by inclusion, and set $\cI=\{C\in\cL:C\in\mathcal K\}$. If $\mathcal K$ is closure stable, then $\cI$ is a lower ideal in $\cL$, and
\[
\mathcal K\simeq\Delta(\cI\setminus\{\varnothing\}).
\]
\end{lemma}
\begin{proof}
If $C\in\cI$ and $D\in\cL$ with $D\subseteq C$, then $D\in\mathcal K$ because $\mathcal K$ is simplicial. Since $D$ is closed, $D\in\cI$. Thus $\cI$ is a lower ideal. Let $\mathcal F(\mathcal K)$ be the poset of nonempty faces of $\mathcal K$, and define
\[
\phi\colon\mathcal F(\mathcal K)
\longrightarrow
\cI\setminus\{\varnothing\},
\qquad
\phi(\sigma)=\cl(\sigma).
\]
This map is well defined. Idempotence makes $\cl(\sigma)$ closed, closure stability gives $\cl(\sigma)\in\mathcal K$, and extensivity shows that $\cl(\sigma)$ is nonempty. Monotonicity makes $\phi$ order preserving. For $C\in\cI\setminus\{\varnothing\}$, the inverse image of the lower ideal below $C$ is
\[
\begin{aligned}
\phi^{-1}(\cI_{\le C})
&=
\{\sigma\in\mathcal F(\mathcal K):
\cl(\sigma)\subseteq C\}\\
&=
2^C\setminus\{\varnothing\}.
\end{aligned}
\]
Indeed, if $\cl(\sigma)\subseteq C$, then $\sigma\subseteq C$ by extensivity. Conversely, if $\varnothing\ne\sigma\subseteq C$, then $\sigma\in\mathcal K$ because $C\in\mathcal K$ and $\mathcal K$ is simplicial, while $\cl(\sigma)\subseteq\cl(C)=C$ by monotonicity and the closedness of $C$. The order complex of $2^C\setminus\{\varnothing\}$ is the barycentric subdivision of the simplex on $C$, and hence is contractible. Quillen's fiber theorem \cite[Theorem~A]{Quillen} therefore gives $\Delta(\mathcal F(\mathcal K))\simeq\Delta(\cI\setminus\{\varnothing\})$. Finally, $\Delta(\mathcal F(\mathcal K))$ is the barycentric subdivision of $\mathcal K$, and therefore is homeomorphic to $\mathcal K$. Hence $\mathcal K\simeq\Delta(\cI\setminus\{\varnothing\})$.
\end{proof}

\begin{lemma}\label[lemma]{lem:3}
Let $C$ be a nonempty closed set in a finite convex geometry $(E,\cl)$, and let $\cL$ be the lattice of closed sets. Then
\[
\Delta((\varnothing,C)_{\cL})\simeq
\begin{cases}
\mathbb{S}^{|C|-2},& C\text{ is free},\\
\text{a contractible complex},& C\text{ is not free}.
\end{cases}
\]
For $|C|=1$, the first line uses the convention $\mathbb{S}^{-1}=\{\varnothing\}$.
\end{lemma}
\begin{proof}
Since $C$ is closed, the interval $[\varnothing,C]_{\cL}$ is the lattice of closed sets of the convex geometry induced on $C$, and hence is meet-distributive \cite[Theorem~3.3]{Edelman}. We first note that $C$ is free if and only if every subset of $C$ is closed. If every subset is closed, then $\cl(C\setminus\{c\})=C\setminus\{c\}$ for each $c\in C$, so every $c$ is extreme. Conversely, suppose $C$ is free and let $X\subseteq C$. If $X$ is not closed, choose $y\in\cl(X)\setminus X$. Since $C$ is closed, $\cl(X)\subseteq\cl(C)=C$, so $y\in C\setminus X$. Hence $X\subseteq C\setminus\{y\}$ and monotonicity gives $y\in\cl(X)\subseteq\cl(C\setminus\{y\})$, contradicting $y\in\ex(C)$. Thus every subset of $C$ is closed. If $C$ is free, it follows that $[\varnothing,C]_{\cL}=2^C$, so this interval is Boolean. Conversely, if $[\varnothing,C]_{\cL}$ is Boolean, then every subset of $C$ is closed, and the preceding argument gives $\ex(C)=C$. When $C$ is free, $(\varnothing,C)_{\cL}$ is therefore the proper part of the Boolean lattice on $C$. Its order complex is the barycentric subdivision of the boundary of a $(|C|-1)$-simplex, so it has homotopy type $\mathbb{S}^{|C|-2}$. For $|C|=1$, this is $\mathbb{S}^{-1}=\{\varnothing\}$. If $C$ is not free, then $[\varnothing,C]_{\cL}$ is a non-Boolean meet-distributive lattice. By \cite[Corollary~2.2]{BilleraHsiaoProvan}, the order complex of its proper part is a polyhedral ball, hence is contractible.
\end{proof}

The result below concerns closure stable complexes. It shows that nonzero reduced homology in degree $q$ gives a free closed nonface of cardinality $q+2$.

\begin{theorem}\label{thm:2}
Let $(E,\cl)$ be a finite convex geometry and let $\mathcal K$ be a simplicial complex on $E$ that is closure stable with respect to $\cl$. If $\widetilde H_q(\mathcal K;\kk)\ne0$ for some $q\ge0$, then there exists a free closed set $B\subseteq E$ such that $B\notin\mathcal K$ and $|B|=q+2$.
\end{theorem}

\begin{proof}
Let $\cL$ and $\cI$ be as in \cref{lem:2}. The lower ideal $\cI$ is proper. Otherwise $E\in\cI$, so $\Delta(\cI\setminus\{\varnothing\})$ would have the top element $E$ and would therefore be a cone, contradicting $\widetilde H_q(\mathcal K;\kk)\ne0$ and \cref{lem:2}. List the closed sets outside $\cI$ as $C_1,\dots,C_M$ in nondecreasing order of cardinality, breaking ties arbitrarily, and set $\mathcal P_0=\cI\setminus\{\varnothing\}$ and $\mathcal P_j=\mathcal P_{j-1}\cup\{C_j\}$. Every nonempty proper closed subset of $C_j$ already lies in $\mathcal P_{j-1}$. It either belongs to $\cI$, or it lies outside $\cI$ and has smaller cardinality. Moreover, no element of $\cI$ properly contains $C_j$, since $\cI$ is a lower ideal, and no previously added set outside $\cI$ can contain $C_j$ because it would have larger cardinality. Thus $C_j$ is maximal in $\mathcal P_j$. Since $C_j$ is the only new element, every new simplex of $\Delta(\mathcal P_j)$ contains the vertex $C_j$, and removing that vertex leaves a chain in $(\varnothing,C_j)_{\cL}$. Conversely, adjoining $C_j$ to any such chain gives a new simplex. If $L_j=\Delta((\varnothing,C_j)_{\cL})$, then $\Delta(\mathcal P_j)$ is obtained from $\Delta(\mathcal P_{j-1})$ by attaching the cone $C_j*L_j$, and the intersection of this cone with the old complex is exactly its base $L_j$. Therefore, by excision, equivalently by passing to the quotient by $\Delta(\mathcal P_{j-1})$,
\[
\widetilde H_t(\Delta(\mathcal P_j),\Delta(\mathcal P_{j-1});\kk)
\cong
\widetilde H_t(C_j*L_j,L_j;\kk).
\]
With the homology conventions above, the relative chain complex of $(C_j*L_j,L_j)$ is obtained from the reduced chain complex of $L_j$ by shifting degrees up by one. Hence
\[
\widetilde H_t(C_j*L_j,L_j;\kk)
\cong
\widetilde H_{t-1}(L_j;\kk).
\]
Hence
\[
\widetilde H_t(\Delta(\mathcal P_j),\Delta(\mathcal P_{j-1});\kk)
\cong
\widetilde H_{t-1}
\bigl(\Delta((\varnothing,C_j)_{\cL});\kk\bigr).
\]
By \cref{lem:3}, this group vanishes in every degree when $C_j$ is not free. If $C_j$ is free, then $L_j\simeq \mathbb{S}^{|C_j|-2}$, so its reduced homology is concentrated in degree $|C_j|-2$. Thus the relative group can be nonzero only when $t=|C_j|-1$. Suppose that no free closed set outside $\cI$ has cardinality $q+2$. By the argument above, the relative group in degree $q+1$ can be nonzero only when $C_j$ is free and $|C_j|-1=q+1$, that is, only when $|C_j|=q+2$. Hence the relative group in degree $q+1$ vanishes for every $j$. The long exact sequence of the pair $(\Delta(\mathcal P_j),\Delta(\mathcal P_{j-1}))$ contains
\[
\widetilde H_{q+1}
(\Delta(\mathcal P_j),\Delta(\mathcal P_{j-1});\kk)
\longrightarrow
\widetilde H_q(\Delta(\mathcal P_{j-1});\kk)
\longrightarrow
\widetilde H_q(\Delta(\mathcal P_j);\kk).
\]
Since the first group is zero, exactness makes the second map injective. By \cref{lem:2}, the initial $q$-th homology is nonzero. Composing these injections, this class survives in the $q$-th homology of $\Delta(\mathcal P_M)$. Every nonempty closed set is either already in $\cI$ or occurs among $C_1,\dots,C_M$, so $\mathcal P_M=\cL\setminus\{\varnothing\}$. This poset has the top element $E$, and therefore its order complex is a cone, whose reduced $q$-th homology vanishes, a contradiction. Thus there exists a free closed set $B\notin\cI$ with $|B|=q+2$. Since $B$ is closed, the definition of $\cI$ gives $B\notin\mathcal K$. Therefore $B$ is the required free closed nonface.
\end{proof}

\section{Co-chordal graphs and clique trees}\label{sec:4}

Throughout this section, $H=G^c$ is chordal. We first consider the case $r\le0$, where the degree complex reduces to a clique complex of a chordal graph.

\begin{lemma}\label[lemma]{lem:4}
Assume that $H=G^c$ is chordal, $r\le0$, and $\widetilde H_q(\Delta_{\mathbf a}(J);\kk)\ne0$. Then $q\in\{-1,0\}$ and
\[
|\mathbf a|+\ell\le2s-1.
\]
\end{lemma}

\begin{proof}
Since the reduced homology is nonzero, $\Delta_{\mathbf a}(J)$ is nonvoid. Thus \cref{eq:4} implies that $N$ is contained in a maximal clique of $H$. If $Q\supseteq N$, then $\sum_{j\in Q\setminus N}b_j\ge0\ge r$, so every such maximal clique satisfies \cref{eq:7}. Hence $\Delta_{\mathbf a}(J)=\operatorname{lk}_{\operatorname{Cl}(H)}(N)$. Indeed, a set $F\subseteq[n]\setminus N$ lies in the complex generated by $Q\setminus N$ with $Q\supseteq N$ exactly when $F\cup N$ is contained in a maximal clique of $H$. Since every clique of a finite graph is contained in a maximal clique, this is equivalent to $F\cup N$ being a clique of $H$, which is precisely the definition of the link. Let $L$ be the induced subgraph of $H$ on the vertices $v\notin N$ for which $N\cup\{v\}$ is a clique. Then $\operatorname{lk}_{\operatorname{Cl}(H)}(N)=\operatorname{Cl}(L)$. Since induced subgraphs of chordal graphs are chordal, $L$ is chordal. Each connected component of $\operatorname{Cl}(L)$ is contractible. To see this, let $L_0$ be a connected component and argue by induction on $|V(L_0)|$. The claim is clear when $L_0$ has one vertex. Otherwise, since $L_0$ is chordal, it has a simplicial vertex $v$. Its neighbor set $N_{L_0}(v)$ is a nonempty clique because $L_0$ is connected. Moreover, $L_0-v$ is connected. If a path in $L_0$ uses $v$, the two neighbors of $v$ on that path are adjacent and allow the path to bypass $v$. Thus $L_0-v$ is a connected chordal graph. The star of $v$ in $\operatorname{Cl}(L_0)$ is the simplex on $\{v\}\cup N_{L_0}(v)$, and its intersection with $\operatorname{Cl}(L_0-v)$ is the simplex on $N_{L_0}(v)$. Hence the star collapses onto this face while $\operatorname{Cl}(L_0-v)$ is fixed. Therefore $\operatorname{Cl}(L_0)$ collapses onto $\operatorname{Cl}(L_0-v)$, which is contractible by induction. Therefore $\operatorname{Cl}(L)$ has reduced homology only in degrees $-1$ and $0$, and hence $q\in\{-1,0\}$. Since $r=W-s+1\le0$, we have $W\le s-1$. By \cref{eq:6}, $|\mathbf a|+\ell=W+q+1\le s\le2s-1$, with the stronger bound $|\mathbf a|+\ell\le s-1$ when $q=-1$.
\end{proof}

We now introduce the clique tree model used for the case $r\ge1$. Throughout this section, we identify a tree or subtree with its vertex set whenever no confusion can arise. A clique tree of a connected chordal graph is a tree whose nodes are the maximal cliques and such that, for each vertex $v$ of the graph, the maximal cliques containing $v$ induce a connected subtree \cite{BlairPeyton}. This is equivalent to the usual subtree representation of chordal graphs \cite{Gavril}. If $H$ is disconnected, choose a clique tree in each component and join these trees by auxiliary edges so that their union is a tree. Since every vertex of $H$ belongs to one component, its occurrence set remains connected in the resulting tree.

Fix $N\subseteq[n]$ that is contained in a maximal clique, and let $T$ be such a tree on $\cQ(H)$. For $v\in[n]$, set
\[
T_v=\{t\in V(T):v\in Q_t\},
\]
where $Q_t\in\cQ(H)$ is the maximal clique labelling $t$. Each $T_v$ is a subtree. Define $T_N:=\bigcap_{v\in N}T_v$, with $T_N=T$ when $N=\varnothing$. The set $T_N$ is nonempty because a maximal clique containing $N$ gives a node belonging to every $T_v$, $v\in N$. As a nonempty intersection of subtrees of a tree, $T_N$ is itself a subtree.

For $j\notin N$, put $P_j=T_j\cap T_N$. Thus $P_j$ is empty or a subtree of $T_N$. Give $P_j$ weight $b_j$, and for $t\in V(T_N)$ define its load by
\[
\lambda(t)=\sum_{\substack{j\notin N\\ t\in P_j}}b_j.
\]
Since $t\in T_N$, we have $N\subseteq Q_t$. Moreover, for $j\notin N$, $t\in P_j$ if and only if $t\in T_j$, equivalently $j\in Q_t$. Hence $\lambda(t)=\sum_{j\in Q_t\setminus N}b_j$. By \cref{eq:7}, the selected nodes are
\[
A=\{t\in V(T_N):\lambda(t)\ge r\}.
\]

We use the following form of tree convexity.

\begin{lemma}\label[lemma]{lem:5}
Let $T$ be a finite tree and let $A\subseteq V(T)$. For $X\subseteq A$, set $\cl_A(X)=A\cap\conv_T(X)$, where, for $X\ne\varnothing$, $\conv_T(X)$ denotes the minimal subtree of $T$ containing $X$, and $\conv_T(\varnothing)=\varnothing$. Then $\cl_A$ is a convex geometry closure on $A$. If $F$ is a subtree of $T$ and $X\subseteq F\cap A$, then $\cl_A(X)\subseteq F\cap A$. Consequently, every simplicial complex generated by traces $F\cap A$ of subtrees is closure stable.
\end{lemma}

\begin{proof}
By definition, $X\subseteq\conv_T(X)$, so the operator is extensive. If $X\subseteq Y$, then $\conv_T(Y)$ is a connected subtree containing $X$. The minimality of $\conv_T(X)$ therefore gives $\conv_T(X)\subseteq\conv_T(Y)$, so the operator is monotone. Finally, $\conv_T(X)$ is itself a connected subtree. Hence the minimal subtree containing $\conv_T(X)$ is $\conv_T(X)$ itself, and therefore $\conv_T\bigl(\conv_T(X)\bigr)=\conv_T(X)$. Thus the operator is idempotent. For anti-exchange, suppose first that $X=\varnothing$. If $x\ne y$, then $x\notin\conv_T(\{y\})=\{y\}$, so the assertion is immediate. Now let $X\ne\varnothing$, put $C=\conv_T(X)$, and suppose that distinct $x,y\notin C$ satisfy $x\in\conv_T(X\cup\{y\})$. Since $C$ is connected, $\conv_T(X\cup\{y\})$ is the union of $C$ and the unique path joining $y$ to $C$. Therefore $x$ lies on this path. Since $x\ne y$, the path joining $x$ to $C$ does not contain $y$, and hence $y\notin\conv_T(X\cup\{x\})$. Thus $\conv_T$ is a convex geometry closure operator. Since the restriction of a finite convex geometry is again a convex geometry, $\cl_A$ is also a convex geometry closure operator. Finally, let $X$ be a face of a complex generated by subtree traces. Then $X\subseteq F\cap A$ for some subtree $F$. Since $F$ is a connected subtree containing $X$, the minimality of $\conv_T(X)$ gives $\conv_T(X)\subseteq F$. Consequently, $\cl_A(X)=\conv_T(X)\cap A\subseteq F\cap A$. Thus $\cl_A(X)$ is again a face of the same trace, and every complex generated by such traces is closure stable.
\end{proof}

We apply this tree closure to $T_N$. Thus, on the selected node set $A\subseteq V(T_N)$, we use $\cl_A(X)=A\cap\conv_{T_N}(X)$ for $X\subseteq A$.

\begin{proposition}\label[proposition]{prop:2}
Assume that $H=G^c$ is chordal, $r\ge1$, and that $\Delta_{\mathbf a}(J)$ is nonvoid. Then
\[
\Delta_{\mathbf a}(J)\simeq
\mathcal K_A:=\left\langle P_j\cap A:j\notin N\right\rangle,
\]
where $\mathcal K_A$ is a simplicial complex on the selected node set $A$.
\end{proposition}

\begin{proof}
For $t\in A$, let $U_t=2^{Q_t\setminus N}$, the full simplex on $Q_t\setminus N$. The selected nodes are exactly the maximal cliques that satisfy \cref{eq:7}, so \cref{prop:1} gives $\Delta_{\mathbf a}(J)=\bigcup_{t\in A}U_t$. Because the degree complex is nonvoid, $A\ne\varnothing$. If $t\in A$, then $\sum_{j\in Q_t\setminus N}b_j=\lambda(t)\ge r\ge1$. All $b_j$ are nonnegative, so $Q_t\setminus N\ne\varnothing$. Thus $\{U_t:t\in A\}$ is a nonempty cover by simplices with nonempty geometric realizations. For every nonempty $\sigma\subseteq A$,
\[
\bigcap_{t\in\sigma}U_t
=
2^{\,\bigcap_{t\in\sigma}(Q_t\setminus N)}.
\]
Hence every geometrically nonempty intersection is a simplex and therefore contractible. The nerve lemma \cite[Theorem~10.6]{Bjorner} gives $\Delta_{\mathbf a}(J)\simeq\operatorname{Nerve}\{U_t:t\in A\}$. A nonempty set $\sigma\subseteq A$ is a face of this nerve exactly when $\bigcap_{t\in\sigma}(Q_t\setminus N)\ne\varnothing$, equivalently, when there is $j\notin N$ with $j\in Q_t$ for every $t\in\sigma$. Since $\sigma\subseteq A\subseteq T_N$, the definition of $P_j$ shows that this is the same as $\sigma\subseteq P_j\cap A$ for some $j\notin N$. Thus the nerve is $\langle P_j\cap A:j\notin N\rangle=\mathcal K_A$.
\end{proof}

Empty $P_j$ may be omitted from the generating family, while each nonempty $P_j$ is a subtree of $T_N$. Thus \cref{lem:5} shows that the trace complex $\mathcal K_A$ is closure stable with respect to $\cl_A$.

For a nonempty finite tree $U$, let $\operatorname{Leaf}(U)$ denote the set of vertices of degree one, with the unique vertex of a singleton tree also regarded as a leaf. The next lemma identifies the extreme elements for tree convexity and connects the homological witness in \cref{thm:2} with the leaf sets used below.

\begin{lemma}\label[lemma]{lem:6}
Let $T$ be a finite tree, let $A\subseteq V(T)$, and define $\cl_A(X)=A\cap\conv_T(X)$ for $X\subseteq A$. For every nonempty $B\subseteq A$, $\ex_A(B)=\operatorname{Leaf}(\conv_T(B))$. Consequently, if $B$ is closed with respect to $\cl_A$, then $B$ is free if and only if $B=\operatorname{Leaf}(\conv_T(B))$.
\end{lemma}

\begin{proof}
Set $U=\conv_T(B)$. Suppose first that $B=\{b\}$. Then $U=\{b\}$, so our convention gives $\operatorname{Leaf}(U)=\{b\}$. Moreover, $b\notin\cl_A(\varnothing)=\varnothing$, and hence $\ex_A(B)=\{b\}$. Thus the assertion holds in this case. Now assume that $|B|\ge2$. For every $b\in B$, since $b\in A$, we have $b\in\ex_A(B)$ if and only if $b\notin\cl_A(B\setminus\{b\})$, which is equivalent to $b\notin\conv_T(B\setminus\{b\})$. Every leaf of $U$ belongs to $B$. Indeed, if a leaf $\ell$ of $U$ did not belong to $B$, then $U\setminus\{\ell\}$ would be a connected subtree containing $B$, contradicting the minimality of $U=\conv_T(B)$. If $b\in B$ is not a leaf of $U$, then $b$ lies on the path joining two leaves of $U$. These leaves belong to $B\setminus\{b\}$, and therefore $b\in\conv_T(B\setminus\{b\})$. Thus $b$ is not extreme. Conversely, if $b$ is a leaf of $U$, then $U\setminus\{b\}$ is a connected subtree containing $B\setminus\{b\}$. Consequently, $\conv_T(B\setminus\{b\})\subseteq U\setminus\{b\}$, so $b\notin\conv_T(B\setminus\{b\})$. Hence $b$ is extreme. It follows that $\ex_A(B)=\operatorname{Leaf}(U)$. If $B$ is closed with respect to $\cl_A$, the final assertion follows because $B$ is free precisely when $\ex_A(B)=B$.
\end{proof}

We now convert a leaf witness into a numerical estimate by a degree count in a tree.

\begin{lemma}\label[lemma]{lem:7}
Let $U$ be a finite tree whose leaf set is $B$, with $|B|\ge2$. For $x\in V(U)\setminus B$, put $\mu(x)=\deg_U(x)-2$. If $R$ is a nonempty proper connected subtree of $U$ and $k(R)=|R\cap B|$, then
\begin{equation}\label{eq:8}
\sum_{x\in R\setminus B}\mu(x)
=k(R)+|\delta_U(R)|-2
\ge k(R)-1,
\end{equation}
where $\delta_U(R)$ is the set of edges with exactly one endpoint in $R$.
\end{lemma}

\begin{proof}
Since $R$ is a tree, the handshaking lemma gives $\sum_{x\in R}\deg_R(x)=2(|R|-1)$. Each edge of $\delta_U(R)$ contributes one additional unit to $\sum_{x\in R}\deg_U(x)$, and hence $\sum_{x\in R}\deg_U(x)=2(|R|-1)+|\delta_U(R)|$. Subtracting $2|R|$ gives
\[
\sum_{x\in R}(\deg_U(x)-2)=|\delta_U(R)|-2.
\]
Each vertex of $R\cap B$ is a leaf of $U$ and contributes $-1$. Moving these $k(R)$ contributions to the right gives the equality in \cref{eq:8}. Since $R$ is nonempty and proper, $|\delta_U(R)|\ge1$, which gives the inequality.
\end{proof}

\begin{corollary}\label[corollary]{cor:1}
Let $T$ be a finite tree, let $A\subseteq V(T)$, and let $\mathcal K=\langle F\cap A:F\in\mathcal F\rangle$ for a finite family $\mathcal F$ of subtrees. If $\widetilde H_q(\mathcal K;\kk)\ne0$ for some $q\ge0$, then there is a subtree $U\subseteq T$ whose leaf set $B$ satisfies $|B|=q+2$ and $U\cap A=B$, and no $F\in\mathcal F$ contains all of $B$.
\end{corollary}

\begin{proof}
By \cref{lem:5}, the trace complex $\mathcal K$ is closure stable with respect to $\cl_A$, so \cref{thm:2} applies. Hence there is a free closed nonface $B\subseteq A$ with $|B|=q+2$. Set $U=\conv_T(B)$. Since $B$ is closed, $U\cap A=\cl_A(B)=B$. Since $B$ is free, \cref{lem:6} shows that $B$ is the leaf set of $U$. If some $F\in\mathcal F$ contained $B$, then $B\subseteq F\cap A$ would be a face of $\mathcal K$, contradicting the choice of $B$.
\end{proof}

\begin{theorem}\label{thm:3}
Let $T$ be a finite tree, let $P_1,\dots,P_h$ be nonempty subtrees with positive integer weights $w_1,\dots,w_h$, and set $\Omega=\sum_{j=1}^h w_j$ and $\lambda(t)=\sum_{1\le j\le h,\,t\in P_j}w_j$. For an integer $r\ge1$, define $A=\{t\in V(T):\lambda(t)\ge r\}$, and let $\mathcal K$ be a simplicial complex on $A$ that is closure stable with respect to $\cl_A$ and satisfies $P_j\cap A\in\mathcal K$ for all $j$. If $\widetilde H_q(\mathcal K;\kk)\ne0$ for some $q\ge0$, then
\[
\Omega\ge2r+q.
\]
\end{theorem}

\begin{proof}
Every facet $F$ of $\mathcal K$ is closed. Indeed, closure stability gives $\cl_A(F)\in\mathcal K$, while $F\subseteq\cl_A(F)$. Hence maximality of $F$ yields $\cl_A(F)=F$. Thus $F=A\cap\conv_T(F)$. Consequently, $\mathcal K$ is generated by traces of subtrees of $T$. Since $P_j\cap A\in\mathcal K$ for every $j$, we may include $P_1,\dots,P_h$ among these subtrees without changing $\mathcal K$. By \cref{cor:1}, there is a subtree $U\subseteq T$ whose leaf set $B$ satisfies $m:=|B|=q+2$ and $U\cap A=B$, and no $P_j$ contains all of $B$. Since $q\ge0$, we have $m\ge2$. Now every $b\in B$ belongs to $A$, so $\lambda(b)\ge r$, while every $x\in V(U)\setminus B$ lies outside $A$, and hence $\lambda(x)\le r-1$. Set $R_j=P_j\cap U$ and $k_j=|P_j\cap B|$. If $R_j$ is nonempty, it is connected as an intersection of subtrees, and it is proper because $P_j$ does not contain all leaves of $U$. Hence \cref{eq:8} gives $k_j\le1+\sum_{x\in R_j\setminus B}\mu(x)$. If $R_j$ is empty, then $k_j=0$, so the same inequality is immediate. Multiplying by $w_j$ and summing gives
\[
\sum_{j=1}^h w_jk_j
\le \Omega+
\sum_{j=1}^h w_j\sum_{x\in R_j\setminus B}\mu(x).
\]
We now identify both sums by double counting. Since $k_j=|P_j\cap B|$,
\[
\begin{aligned}
\sum_{b\in B}\lambda(b)
&=\sum_{b\in B}\sum_{\substack{1\le j\le h\\ b\in P_j}}w_j\\
&=\sum_{j=1}^h w_j|P_j\cap B|
 =\sum_{j=1}^h w_jk_j.
\end{aligned}
\]
Also, for $x\in V(U)$ we have $x\in R_j$ exactly when $x\in P_j$. Therefore
\[
\begin{aligned}
\sum_{j=1}^h w_j\sum_{x\in R_j\setminus B}\mu(x)
&=\sum_{x\in V(U)\setminus B}\mu(x)
  \sum_{\substack{1\le j\le h\\ x\in P_j}}w_j\\
&=\sum_{x\in V(U)\setminus B}\mu(x)\lambda(x).
\end{aligned}
\]
Since every $b\in B$ has load at least $r$, these identities give
\[
rm
\le \sum_{b\in B}\lambda(b)
=\sum_{j=1}^h w_jk_j
\le \Omega+
\sum_{x\in V(U)\setminus B}\mu(x)\lambda(x).
\]
Since $B$ is the full leaf set of $U$, every vertex $x\in V(U)\setminus B$ has degree at least two, and therefore $\mu(x)=\deg_U(x)-2\ge0$. Using $\lambda(x)\le r-1$ for $x\in V(U)\setminus B$, we obtain
\[
rm\le\Omega+(r-1)
\sum_{x\in V(U)\setminus B}\mu(x).
\]
Finally, the degree sum formula for a tree gives $\sum_{x\in V(U)}(\deg_U(x)-2)=-2$. Each of the $m$ leaves contributes $-1$, so $\sum_{x\in V(U)\setminus B}\mu(x)=m-2$. Therefore $rm\le\Omega+(r-1)(m-2)$, and hence $\Omega\ge2r+m-2=2r+q$.
\end{proof}

\begin{example}\label[example]{ex:1}
The bound in \cref{thm:3} is sharp for every $r\ge1$ and $q\ge0$. Put $m=q+2$, and let $T$ be the star with center $c$ and leaves $B=\{b_1,\dots,b_m\}$. Take $r$ copies of $\{b_1\}$, one copy of $\{b_j\}$ for each $2\le j\le m$, and $r-1$ copies of $T\setminus\{b_1\}$, all with unit weight. Then
\[
\Omega=r+(m-1)+(r-1)=2r+q.
\]
Moreover, $\lambda(b_j)=r$ for every $1\le j\le m$, while $\lambda(c)=r-1$. Hence the selected set in \cref{thm:3} is $A=B$. Choose the simplicial complex $\mathcal K = \left\langle (T\setminus\{b_j\})\cap A:1\le j\le m\right\rangle$. Since $A=B$,
\[
\mathcal K
=
\left\langle B\setminus\{b_j\}:1\le j\le m\right\rangle,
\]
which is the boundary complex of the simplex on $B$. Therefore $\mathcal K\simeq \mathbb{S}^{m-2}=\mathbb{S}^q$, and hence $\widetilde H_q(\mathcal K;\kk)\ne0$. The singleton traces are faces of $\mathcal K$, while $(T\setminus\{b_1\})\cap A=B\setminus\{b_1\}$ is a facet. Hence the trace of every weighted subtree is a face of $\mathcal K$. Finally, for every $X\subseteq B$, we have $\conv_T(X)\cap B=X$, so $\cl_A$ is the identity on $A=B$. Hence $\mathcal K$ is closure stable. Thus all the hypotheses of \cref{thm:3} are satisfied with $\Omega=2r+q$, proving that the bound is sharp.
\end{example}

\section{Proof of the main theorem and consequences}\label{sec:5}

\begin{proof}[Proof of \cref{thm:1}]
The lower bound is \cref{lem:1}. Put $J=I(G)^{(s)}$. It remains to prove $\reg S/J\le2s-1$. Let $H^\ell_{\m}(S/J)_{\mathbf a}\ne0$, and set $q=\ell-|N|-1$. By Takayama's formula \cref{eq:2}, $\widetilde H_q(\Delta_{\mathbf a}(J);\kk)\ne0$. In particular, the degree complex is nonvoid. Normalize the negative coordinates of $\mathbf a$ as above and use the parameters in \cref{eq:5}. This does not change $N$ or the degree complex. If $r\le0$, \cref{lem:4} gives $|\mathbf a|+\ell\le2s-1$. Assume $r\ge1$. Since reduced homology vanishes below degree $-1$, we have $q\ge-1$. If $q=-1$, nonzero reduced homology would force $\Delta_{\mathbf a}(J)=\{\varnothing\}$. This complex is nonvoid, so \cref{prop:1} gives a maximal clique $Q$ satisfying \cref{eq:7}. Hence $\sum_{j\in Q\setminus N} b_j\ge r\ge1$, so $Q\setminus N$ is nonempty and gives a nonempty face, a contradiction. Thus $q\ge0$. Now \cref{prop:2} applies and gives $\Delta_{\mathbf a}(J)\simeq\mathcal K_A$. Hence, by homotopy invariance of reduced homology,
\[
\widetilde H_q(\mathcal K_A;\kk)\cong\widetilde H_q(\Delta_{\mathbf a}(J);\kk)\ne0.
\]
Apply \cref{thm:3} to the nonempty subtrees $P_j$ with $b_j>0$, using weights $b_j$. This family is nonempty. Indeed, since $\Delta_{\mathbf a}(J)$ is nonvoid, \cref{prop:1} gives a maximal clique $Q$ satisfying \cref{eq:7}. Hence the corresponding node belongs to $A$, so $A\ne\varnothing$. Choose $t\in A$. Then $\lambda(t)\ge r\ge1$, so $t\in P_j$ for some $j\notin N$ with $b_j>0$. These weights are positive integers because $\mathbf a\in\ZZ^n$, and each corresponding trace $P_j\cap A$ is a face of $\mathcal K_A$ by definition. Empty subtrees and zero weights are omitted only from the weighted family. Their traces remain part of $\mathcal K_A$. Omitting them does not change any load and therefore does not change the selected set $A$. If $W_0$ denotes the sum of these weights, then $W_0=\Omega$, and \cref{thm:3} gives $W\ge W_0=\Omega\ge2r+q$. Since $r=W-s+1$ by \cref{eq:5}, it follows that $W+q\le2s-2$. Using \cref{eq:6}, $|\mathbf a|+\ell=W+q+1\le2s-1$. Thus every nonzero multigraded component of local cohomology satisfies this bound, and \cref{eq:3} gives $\reg S/J\le2s-1$. Since $J\ne0$ is a proper monomial ideal, the exact sequence
\[
0\longrightarrow J\longrightarrow S\longrightarrow S/J\longrightarrow0
\]
gives $\reg J=\reg S/J+1$. Hence $\reg J\le2s$, which together with \cref{lem:1} yields
\[
\reg I(G)^{(s)}=2s.
\]
\end{proof}

The ordinary powers satisfy the same formula by \cite[Theorem~3.2]{HHZ}. Hence
\[
\reg I(G)^{(s)}=\reg I(G)^s=2s
\qquad\text{for every }s\ge1,
\]
which proves Minh's conjecture for co-chordal graphs.

\begin{corollary}\label[corollary]{cor:2}
Let $G$ be a finite simple graph with at least one edge. The following conditions are equivalent.
\begin{enumerate}[label=\textup{(\roman*)}]
\item $G$ is co-chordal.
\item $I(G)^{(s)}$ has a degree resolution for every $s\ge1$.
\end{enumerate}
\end{corollary}

\begin{proof}
If $G$ is co-chordal, \cref{lem:1} gives a minimal generator of degree $2s$, while \cref{thm:1} gives $\reg I(G)^{(s)}=2s$. Since every minimal generator has degree at most the regularity, $\omega(I(G)^{(s)})=2s=\reg I(G)^{(s)}$. Thus every symbolic power has a degree resolution. Conversely, apply the hypothesis with $s=1$. Since $I(G)$ is generated in degree two, a degree resolution gives $\reg I(G)=2$. An ideal generated in a single degree $d$ has a $d$-linear resolution exactly when its regularity is $d$. Hence $I(G)$ has a $2$-linear resolution, and Fr\"oberg's theorem \cite[Theorem~1]{Froberg} implies that $G$ is co-chordal.
\end{proof}

\begin{corollary}\label[corollary]{cor:3}
Let $G$ be a finite simple co-chordal graph with at least one edge, let $s\ge1$, and put $J=I(G)^{(s)}$. For every $d\ge2s$, the truncation $J_{\ge d}=\bigoplus_{k\ge d}J_k$ has a $d$-linear resolution.
\end{corollary}

\begin{proof}
By \cref{thm:1}, $\reg J=2s$. Let $d\ge2s$. By \cite[Theorem~4.1]{AhmedFrobergNamiq}, the graded Betti numbers of $S/J_{\ge d}$ can occur only in degrees $j=i+d-1$ for $i>0$, since $d\ge\reg J$. Hence $J_{\ge d}$ has a $d$-linear resolution.
\end{proof}

We now prove \cref{thm:1.2} and identify the degrees not covered by it. For the remainder of this section, $H=G^c$ is chordal, $J=I(G)^{(s)}$, and, for $Q\subseteq[n]$, we write $\mathbf a(Q)=\sum_{i\in Q}a_i$.

\begin{proposition}\label[proposition]{prop:3}
Let $G$ be co-chordal, let $s\ge1$, and put $J=I(G)^{(s)}$. The ideal $J_{\langle s+1\rangle}$ is either zero or has linear quotients. In particular, whenever it is nonzero, it has an $(s+1)$-linear resolution.
\end{proposition}

\begin{proof}
Let $H=G^c$ and let $\mathbf a\in\NN^n$ satisfy $|\mathbf a|=s+1$. By \cref{eq:1}, $x^{\mathbf a}\in J$ if and only if $\sum_{i\notin Q}a_i\ge s$ for every $Q\in\cQ(H)$. Since $|\mathbf a|=s+1$, this is equivalent to $\mathbf a(Q)\le1$ for every maximal clique $Q$ of $H$. Every vertex belongs to a maximal clique, so $a_i\le1$ for every $i$. Thus every monomial of $J_{\langle s+1\rangle}$ is squarefree, and its support meets each clique of $H$ in at most one vertex. Equivalently,
\begin{equation}\label{eq:9}
J_{\langle s+1\rangle}
=
\bigl(x_A:A\subseteq[n],\ |A|=s+1,\ A\text{ is independent in }H\bigr).
\end{equation}

Choose a perfect elimination ordering $1,\ldots,n$ of $H$ and order the generators in \cref{eq:9} decreasingly in the lexicographic order induced by $x_1>\cdots>x_n$. Let $x_A>x_B$ be two generators, and let $i$ be the first index at which the characteristic vectors of $A$ and $B$ differ. Then $i\in A\setminus B$. If $k<i$ and $k\in B$, then $k\in A$, so $k$ is not adjacent to $i$ in $H$. The later neighbors of $i$ form a clique, while $B$ is independent, so $B$ contains at most one later neighbor of $i$. If such a neighbor exists, call it $j$. It necessarily belongs to $B\setminus A$. If it does not exist, choose any $j\in B\setminus A$. In either case $C=(B\setminus\{j\})\cup\{i\}$ is independent in $H$. Also $j>i$, so $x_C>x_B$, $(x_C:x_B)=x_i$, and $x_i\mid(x_A:x_B)$. This is the linear quotient criterion for the above lexicographic ordering. Since all generators have degree $s+1$, the resolution is $(s+1)$-linear.
\end{proof}

\begin{proof}[Proof of \cref{thm:1.2}]
Suppose that $x^{\mathbf a}\in J$ has degree $d$. Since $J$ is proper, $d>0$. Choose $i$ with $a_i>0$, and let $Q$ be a maximal clique of $H$ containing $i$. By \cref{eq:1}, $s\le \sum_{j\notin Q}a_j=d-\mathbf a(Q)$, while $\mathbf a(Q)\ge a_i\ge1$. Hence $d\ge s+1$, proving that $J_{\langle d\rangle}=0$ for $d\le s$.

Part~\textup{(ii)} is \cref{prop:3}. If $d\ge2s$, then \cref{cor:3} shows that $J_{\ge d}$ has a $d$-linear resolution and is therefore generated by its degree-$d$ component. Consequently, $J_{\langle d\rangle}=J_{\ge d}$, which proves~\textup{(iii)}. The only degrees not covered are therefore $s+2\le d\le2s-1$, and this interval is nonempty only when $s\ge3$.
\end{proof}

For $s=2$, Theorem~\ref{thm:1.2} recovers the componentwise linearity of $I(G)^{(2)}$, which was proved with linear quotients by Ficarra, Moradi, and R\"omer \cite[Theorem~3.5]{FMR}. The next result shows that componentwise linearity need not hold for higher symbolic powers.

\begin{proposition}\label[proposition]{prop:4}
For every integer $s\ge4$, there is a co-chordal graph $G$ such that $I(G)^{(s)}$ is not componentwise linear. More precisely, the ideal generated by the degree $s+2$ component of $I(G)^{(s)}$ has regularity at least $s+3$.
\end{proposition}
\begin{proof}
Fix $s\ge4$. Let $H_s$ be the $s$-sun on $A=\{a_1,\ldots,a_s\}$ and $B=\{b_1,\ldots,b_s\}$, where $A$ is a clique, $B$ is independent, and $b_i$ is adjacent precisely to $a_i$ and $a_{i+1}$, with indices taken modulo $s$. Let $C_s$ be the cycle on $A$ with edges $\{a_i,a_{i+1}\}$ for $1\le i\le s$. Put $G=H_s^c$, $J=I(G)^{(s)}$, and $L=J_{\langle s+2\rangle}$. Since $H_s$ is chordal, $G$ is co-chordal. The maximal cliques of $H_s$ are $A$ and $Q_i=\{b_i,a_i,a_{i+1}\}$ for $1\le i\le s$.

Let $\boldsymbol{\alpha}\in\NN^{2s}$ be the characteristic vector of $B$, with coordinates indexed by $A\cup B$. Thus $\alpha_{a_i}=0$ and $\alpha_{b_i}=1$ for $1\le i\le s$, so $x^{\boldsymbol{\alpha}}=\prod_{i=1}^s x_{b_i}$. Then $|\boldsymbol{\alpha}|=s$ and $N(\boldsymbol{\alpha})=\varnothing$. Moreover,
\[
\sum_{v\notin A}\alpha_v=s,
\qquad
\sum_{v\notin Q_i}\alpha_v=s-1
\quad\text{for }1\le i\le s.
\]
By \cref{eq:1}, $x^{\boldsymbol{\alpha}}\notin J$, so $\Delta_{\boldsymbol{\alpha}}(J)$ is nonvoid. Hence \cref{prop:1} gives
\[
\Delta_{\boldsymbol{\alpha}}(J)=\langle Q_1,\ldots,Q_s\rangle.
\]
Since $L\subseteq J$, we have $\Delta_{\boldsymbol{\alpha}}(J)\subseteq\Delta_{\boldsymbol{\alpha}}(L)$. For the reverse inclusion, let $F$ be a set not contained in any $Q_i$. Then $F$ contains two distinct vertices $u,v$ that do not lie together in any $Q_i$. Indeed, if every two vertices of $F$ lie together in some $Q_i$ and $b_i\in F$, then $F\subseteq Q_i$, because $b_i$ belongs only to $Q_i$. If $F\subseteq A$, then every two vertices of $F$ are consecutive on $C_s$. Since $s\ge4$, $C_s$ is triangle free, so $|F|\le2$, and hence $F$ is contained in some $Q_i$.

Since no $Q_i$ contains both $u$ and $v$, the monomial $x^{\boldsymbol{\alpha}}x_ux_v$ has degree $s+2$, and its exponent sum on every maximal clique of $H_s$ is at most two. Equivalently, its exponent sum outside every maximal clique is at least $s$, so \cref{eq:1} gives $x^{\boldsymbol{\alpha}}x_ux_v\in J$. Hence $x^{\boldsymbol{\alpha}}x_ux_v\in L$, and since $u,v\in F$,
\[
x^{\boldsymbol{\alpha}}
=\frac{x^{\boldsymbol{\alpha}}x_ux_v}{x_ux_v}
\in LS_F.
\]
Thus $F\notin\Delta_{\boldsymbol{\alpha}}(L)$, and therefore $\Delta_{\boldsymbol{\alpha}}(L)=\langle Q_1,\ldots,Q_s\rangle$.

For each $i$, collapse $Q_i$ along the free edge $\{a_i,b_i\}$, and then collapse the remaining edge $\{b_i,a_{i+1}\}$ along the free vertex $b_i$. These collapses leave precisely $C_s$, so $\Delta_{\boldsymbol{\alpha}}(L)\simeq C_s$. Hence
\[
\widetilde H_1(\Delta_{\boldsymbol{\alpha}}(L);\kk)\cong\kk.
\]
Since $N(\boldsymbol{\alpha})=\varnothing$, \cref{eq:2} gives $H^2_{\m}(S/L)_{\boldsymbol{\alpha}}\ne0$. Therefore \cref{eq:3} yields
\[
\reg S/L\ge|\boldsymbol{\alpha}|+2=s+2.
\]
As in the proof of \cref{thm:1}, $\reg L=\reg S/L+1$, and hence $\reg L\ge s+3$. Thus $L$ does not have an $(s+2)$-linear resolution, so $I(G)^{(s)}$ is not componentwise linear.
\end{proof}

\begin{remark}
Proposition~\ref{prop:4} disproves Conjecture~B of Ficarra, Moradi, and R\"omer \cite{FMR} and shows that symbolic powers of co-chordal edge ideals need not be componentwise linear. The smallest member of the family occurs for $s=4$. Since $H_4\cong H_4^c$, the graph in \cref{fig:4sun} represents the corresponding counterexample.

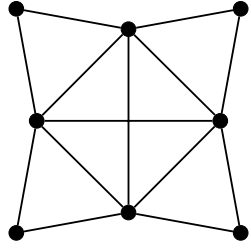
\begin{figure}[ht]
\centering
\begin{tikzpicture}[
scale=.9,
vertex/.style={circle,draw=black,fill=black,minimum size=5.5pt,inner sep=0pt},
edge/.style={line width=0.7pt}
]
\node[vertex] (a1) at (0,1.35) {};
\node[vertex] (a2) at (1.35,0) {};
\node[vertex] (a3) at (0,-1.35) {};
\node[vertex] (a4) at (-1.35,0) {};

\node[vertex] (b1) at (1.65,1.65) {};
\node[vertex] (b2) at (1.65,-1.65) {};
\node[vertex] (b3) at (-1.65,-1.65) {};
\node[vertex] (b4) at (-1.65,1.65) {};

\draw[edge]
(a1)--(a2)
(a2)--(a3)
(a3)--(a4)
(a4)--(a1)
(a1)--(a3)
(a2)--(a4);

\draw[edge]
(b1)--(a1) (b1)--(a2)
(b2)--(a2) (b2)--(a3)
(b3)--(a3) (b3)--(a4)
(b4)--(a4) (b4)--(a1);
\end{tikzpicture}

\caption{The $4$-sun graph $H_4$.}
\label{fig:4sun}
\end{figure}
\end{remark}

\section*{Declarations}

\subsection*{Data Availability}
No data were generated or analyzed in this study.

\subsection*{Competing Interests}
The authors declare no competing interests.

\end{document}